\documentclass[12pt]{amsart}
\usepackage{amsmath,amssymb,amsbsy,amsfonts,amsthm,latexsym,mathabx,
            amsopn,amstext,amsxtra,euscript,amscd,stmaryrd,mathrsfs,
            cite,array,mathtools,enumerate}
            
\usepackage{url}
\usepackage[colorlinks,linkcolor=blue,anchorcolor=blue,citecolor=blue,backref=page]{hyperref}

\usepackage{color}

\renewcommand*{\backref}[1]{}
\renewcommand*{\backrefalt}[4]{%
    \ifcase #1 (Not cited.)%
    \or        (p.\,#2)%
    \else      (pp.\,#2)%
    \fi}
\begin{document}

\newtheorem{theorem}{Theorem}
\newtheorem{lemma}[theorem]{Lemma}
\newtheorem{claim}[theorem]{Claim}
\newtheorem{cor}[theorem]{Corollary}
\newtheorem{prop}[theorem]{Proposition}
\newtheorem{definition}{Definition}
\newtheorem{question}[theorem]{Open Question}
\newtheorem{conj}[theorem]{Conjecture}
\newtheorem{prob}{Problem}
\newtheorem{algorithm}[theorem]{Algorithm}

\def\squareforqed{\hbox{\rlap{$\sqcap$}$\sqcup$}}
\def\qed{\ifmmode\squareforqed\else{\unskip\nobreak\hfil
\penalty50\hskip1em\nobreak\hfil\squareforqed
\parfillskip=0pt\finalhyphendemerits=0\endgraf}\fi}

\numberwithin{equation}{section}
\numberwithin{theorem}{section}

\def\cA{{\mathcal A}}
\def\cB{{\mathcal B}}
\def\cC{{\mathcal C}}
\def\cD{{\mathcal D}}
\def\cE{{\mathcal E}}
\def\cF{{\mathcal F}}
\def\cG{{\mathcal G}}
\def\cH{{\mathcal H}}
\def\cI{{\mathcal I}}
\def\cJ{{\mathcal J}}
\def\cK{{\mathcal K}}
\def\cL{{\mathcal L}}
\def\cM{{\mathcal M}}
\def\cN{{\mathcal N}}
\def\cO{{\mathcal O}}
\def\cP{{\mathcal P}}
\def\cQ{{\mathcal Q}}
\def\cR{{\mathcal R}}
\def\cS{{\mathcal S}}
\def\cT{{\mathcal T}}
\def\cU{{\mathcal U}}
\def\cV{{\mathcal V}}
\def\cW{{\mathcal W}}
\def\cX{{\mathcal X}}
\def\cY{{\mathcal Y}}
\def\cZ{{\mathcal Z}}

\def\fA{{\mathfrak A}}
\def\fJ{{\mathfrak J}}

\def\sssum{\mathop{\sum\!\sum\!\sum}}
\def\ssum{\mathop{\sum\ldots \sum}}

\def\Xm{\cX_m}

\def \C {{\mathbb C}}
\def \F {{\mathbb F}}
\def \L {{\mathbb L}}
\def \K {{\mathbb K}}
\def \N {{\mathbb N}}
\def \R {{\mathbb R}}
\def \Q {{\mathbb Q}}
\def \Z {{\mathbb Z}}

\def\barG{\overline{\cG}}
\def\\{\cr}
\def\({\left(}
\def\){\right)}
\def\fl#1{\left\lfloor#1\right\rfloor}
\def\rf#1{\left\lceil#1\right\rceil}

\newcommand{\pfrac}[2]{{\left(\frac{#1}{#2}\right)}}

\def\rem{\mathrm{\, rem~}}

\def \Prob{{\mathrm {}}}
\def\e{\mathbf{e}}
\def\ep{{\mathbf{\,e}}_p}
\def\epp{{\mathbf{\,e}}_{p^2}}
\def\er{{\mathbf{\,e}}_r}
\def\eM{{\mathbf{\,e}}_M}
\def\eps{\varepsilon}
\def\Res{\mathrm{Res}}
\def\vec#1{\mathbf{#1}}

\def \li {\mathrm {li}\,}

\def\ip{\overline p}
\def\ipd{\ip_d}
\def\iq{\overline q}

\def\e{{\mathbf{\,e}}}
\def\ep{{\mathbf{\,e}}_p}
\def\em{{\mathbf{\,e}}_m}

\def\mand{\qquad\mbox{and}\qquad}

\newcommand{\commF}[1]{\marginpar{%
\begin{color}{red}
\vskip-\baselineskip 
\raggedright\footnotesize
\itshape\hrule \smallskip F: #1\par\smallskip\hrule\end{color}}}

\newcommand{\commI}[1]{\marginpar{%
\begin{color}{blue}
\vskip-\baselineskip 
\raggedright\footnotesize
\itshape\hrule \smallskip I: #1\par\smallskip\hrule\end{color}}}

\title[Products of primes  in 
progressions]{On short  products  of primes in arithmetic 
progressions}
\date{\today}

\author{Igor E.  Shparlinski}

\address{Department of Pure Mathematics, University of New South Wales\\
2052 NSW, Australia.}

\email{igor.shparlinski@unsw.edu.au}

\subjclass[2010]{Primary 11N25; Secondary 11B25, 11L07, 11N36}
 \keywords{residue classes, primes,  sieve method, exponential sums}

\begin{abstract} We  give  several families of reasonably small integers $k, \ell \ge 1$  
and real positive $\alpha, \beta \le 1$, such that
the products  $p_1\ldots p_k s$, where $p_1, \ldots, p_k \le m^\alpha$ are primes and $s \le m^\beta$ is a product of at most 
$\ell$ primes,  represent all reduced residue classes modulo $m$. 
This is a relaxed version of 
the still open question of P.~Erd{\H o}s, A.~M.~Odlyzko and A.~S{\'a}rk{\"o}zy  (1987), 
that corresponds to $k = \ell =1$ (that is, to products of two primes).  In particular,  we improve 
recent results of A.~Walker~(2016).
\end{abstract}
 
\maketitle

\section{Introduction}


Since our knowledge of distribution of primes in short arithemetic 
progressions is rather limited, it is certainly interesting to consider various 
modifications and relaxations of this question. In particular, as one of such relaxations, 
Erd{\H o}s, Odlyzko and S{\'a}rk{\"o}zy~\cite{EOS}
have introduced a question about the distribution of products 
of two small primes in residue classes.  Namely, given an integer 
$m\ge 1$, Erd{\H o}s, Odlyzko and S{\'a}rk{\"o}zy~\cite{EOS}
ask whether all reduced classes $a$ modulo $m$
can be represented as the product 
\begin{equation}
\label{eq:congr}
p_1p_2\equiv a \pmod m
\end{equation}
 of two primes $p_1,p_2\le m$
and prove   a series of conditional results towards this 
under various assumptions about the zero-free regions for
Dirichlet $L$--functions. However, it appears that even the Extended
Riemann Hypothesis  is  not powerful enough to answer the original 
question.

Some more accessible relaxations of this problem 
have been introduced  and studied by Friedlander,   Kurlberg and  Shparlinski~\cite{FrKuSh}, where, in particular, 
the congruence~\eqref{eq:congr}  is considered on average over $a$ and $m$. 
Furthermore, one can find in~\cite{FrKuSh} some results on several ternary modifications of~\eqref{eq:congr} 
such as the congruences
\begin{equation}
\label{eq:3 primes}
p_1(p_2+p_3) \equiv a \pmod m \quad \text{and}\quad p_1p_2(p_3+h) \equiv a \pmod m,
\end{equation}
where $p_1,p_2,p_3 \le x$ are primes and $h$ is a fixed integer. Recently the results
of~\cite{FrKuSh} about the congruences~\eqref{eq:3 primes}
have been improved by Garaev~\cite{Gar1,Gar2}. Furthermore, 
the congruence
$$
p_1p_2+p_2p_3+ p_1 p_3  \equiv a \pmod m
$$
with primes $p_1,p_2,p_3 \le x$ has been studied in~\cite{Bak,FoSh}, with some applications
to the size of largest prime divisor of the bilinear quadratic form $p_1p_2+p_2p_3+ p_1 p_3$. 

Yet another relaxation of the original question 
of~\cite{EOS}  have been introduced in~\cite{Shp},  where one of the components 
of the product on the left hand side is prime and the other one is
almost prime (that is, a product of a small number of primes. 

More precisely, or an integer $\ell \ge 1$ we use $\cP_\ell$ to denote the set of 
integers that are products of at most $\ell$ primes. Thus $\cP=\cP_1$ is the set 
of primes.

Now. for some real positive $x,y \le m$ and integer $\ell \ge 1$, 
we consider the congruence
\begin{equation}
\label{eq:p Pl}
p s \equiv a \pmod m, \quad p \in \cP  \cap [1,x], \  s \in \cP_\ell \cap [1,y], 
\end{equation}
that is. with variables   $p \le x$ which is prime and $1 \le s \le y$ which is a product of at most $\ell$ primes,

\begin{definition}
\label{def:triple} 
We say that a triple $(\ell; \alpha, \beta) \in \N \times \R^2$ is {\it admissible\/}, 
if for any fixed $\varepsilon > 0$ and 
$$x = m^{\alpha+\varepsilon} \mand y = m^{\beta+\varepsilon}
$$ 
the congruence~\eqref{eq:p Pl}
has a solution for any reduced residue class $a$ modulo $m$, provided that $m$ is
large enough, and we denote by $\fA_3$ the set of admissible triples. 
\end{definition}

Thus  the question of~\cite{EOS} is equivalent, apart from the presence of $\varepsilon$,
 to proving that $(1;1,1)$ is an
admissible triple, which seems to be out of reach nowadays.
However,  some families of admissible triples, have been given in~\cite[Theorem~3]{Shp}.
In particular,  it is observed  in~\cite[Section~4]{Shp}
that 
\begin{equation}
\label{eq: 17 0.997}
(17;0.997,0.997) \in \fA_3.
\end{equation}

Walker~\cite{Wal} has recently considered a different variant of this question
and asked about the solvability of the congruence 
\begin{equation}
\label{eq: k primes}
p_1\ldots p_k \equiv a \pmod m, \qquad p_1,\ldots, p_k \in \cP  \cap [1,x], 
\end{equation}
that is, where $p_1,\ldots, p_k \le x$ are primes. 

\begin{definition}
\label{def:pair} 
We say that a pair $(k; \alpha) \in \N \times \R$ is {\it admissible\/}, 
if for any fixed $\varepsilon > 0$ and 
$$
x = m^{\alpha+\varepsilon}  
$$    
the congruence~\eqref{eq: k primes}
has a solution for any reduced residue class $a$ modulo $m$, provided that $m$ is
large enough, and we denote by $\fA_2$ the set of admissible pairs.
\end{definition}

Thus in these settings, the question of~\cite{EOS} is equivalent to proving that $(2;1)$ is 
admissible (again, apart from the presence of $\varepsilon$).

We also introduce a similar definition with respect to prime moduli
\begin{definition}
\label{def:pair Prime} 
We say that a pair $(k; \alpha) \in \N \times \R$ is {\it admissible for primes\/}, 
if for any fixed $\varepsilon > 0$ and 
$$
x = m^{\alpha+\varepsilon}  
$$   
the congruence~\eqref{eq: k primes}
has a solution for any reduced residue class $a$ modulo $m$, provided that $m$ is prime
and large enough, and we denote by $\fA_2^\sharp$ the set of admissible for primes pairs.
\end{definition}

Walker~\cite[Theorem~2]{Wal} has shown that 
\begin{equation}
\label{eq: 6 15/16}
(6; 15/16)  \in \fA_2^\sharp,
\end{equation}
(note that  $15/16 = 0.9375\ldots $), as well as that for any $\eta> 0$
there is an integer $k_\eta$, for which 
\begin{equation}
\label{eq: k0 3/4}
(k_\eta; 3/4+\eta)  \in \fA_2^\sharp,
\end{equation}

However, we note that the claim made in~\cite{Wal} that~\eqref{eq: 6 15/16} is an improvement
over~\eqref{eq: 17 0.997} does not seem to be justified. 
Even ignoring the difference between  
arbitrary and prime moduli $m$, which distinguishes~\eqref{eq: 17 0.997} and~\eqref{eq: 6 15/16}, 
we note that while 
\begin{equation}
\label{eq: Impl3-2}
(\ell; \alpha,  \alpha)   \in \fA_3 \  \Longrightarrow \ (\ell+1; \alpha)   \in \fA_2, 
\end{equation}
the opposite implication is not clear and most likely to be false. 

The main goal of this work is to show that there is an alternative and more efficient approach 
to producing pairs  $(k;\alpha) \in \fA_2 $ with reasonably small $k$ and $\alpha$. 
In particular, we obtain a series of improvements of~\eqref{eq: 6 15/16} and~\eqref{eq: k0 3/4},
see Section~\ref{sec: num ex} for the numerical values.

In fact, given some real positive $x,y \le m$ and an integer $\ell \ge 1$,
 we consider a more general congruence, which includes~\eqref{eq:p Pl} and~\eqref{eq: k primes}
 as special cases: 
\begin{equation}
\label{eq: k l}
p_1\ldots p_k s\equiv a \pmod m, \quad p_1,\ldots, p_k \in \cP  \cap [1,x], \ 
s \in \cP_\ell \cap [1,y], 
\end{equation}
that is,  with  variables 
 $p_1,\ldots, p_k \le x$ which are and $1 \le s \le y$ which  is a product of at most $\ell$ primes.

\begin{definition}
\label{def:quadr} 
We say that a quadruple  $(k,\ell; \alpha, \beta) \in \N^2 \times \R^2$ is {\it admissible\/}, 
if for any fixed $\varepsilon > 0$ and 
$$x = m^{\alpha+\varepsilon} \mand y = m^{\beta+\varepsilon}
$$  the congruence~\eqref{eq:p Pl}
has a solution for any reduced residue class $a$ modulo $m$, provided that $m$ is
large enough, and we denote by $\fA_4$ the set of admissible quadruples. 
\end{definition}

Then aforementioned improvements of~\eqref{eq: 6 15/16} follow as special cases 
from the obvious analogue of~\eqref{eq: Impl3-2} that 
\begin{equation}
\label{eq: Impl4-2}
(k,\ell; \alpha,  \alpha) \in \fA_4 \  \Longrightarrow \ (k+\ell; \alpha)   \in \fA_2.
\end{equation}

Our method is based on  bounds of some exponential sums with reciprocals of 
primes. 
These bounds are then
coupled with the sieve method in the form given by  
Greaves~\cite[Section~5]{Gre2}.

In particular, we use this opportunity to improve slightly the result of~\cite{Shp}
about admissible triples
via a more careful choice of parameters and then we introduce  a new argument
which allows us to produce a large family of  admissible quadruples.
In turn, using~\eqref{eq: Impl4-2} we 
significantly improve the results of Walker~\cite{Wal}. 
For example, we replace $1/4$ with $1/2$ in~\eqref{eq: k0 3/4}, 
and extend it to composite moduli, see~\eqref{eq: k0 1/2}  below. 

Finally, we remark that here all elements of the product are less than the modulus, that is, 
we always have $\alpha, \beta \le 1$. For products of large primes one can achieve more, 
and for example by a result of Ramar{\'e} and   Walker~\cite{RaWa} every reduced class modulo $m$ 
can be represented by a product of three primes $p_1, p_2, p_3 \le m^4$ (provided that $m$ is large enough).

\section{Main result and its implications}
\label{sec:main}

Following the results of Greaves~\cite[Equation~(1.4)]{Gre1}, see also~\cite[pp.~174--175]{Gre2}, 
we also define the constants
\begin{equation}
\label{eq:delta234}
\delta_2 = 0.044560, \qquad \delta_3 = 0.074267, 
\qquad \delta_4 = 0.103974, 
\end{equation}
and, after rounding up, 
$$
\delta_\ell = 0.124821, \qquad \ell \ge 5.
$$
We also define 
$$
\vartheta_\ell = \ell - \delta_\ell, \qquad \ell =2,3, \ldots, 
$$

First we prove the following general statement. 

\begin{theorem}\label{thm:pPkPl}
For any fixed real  $\alpha \ge 1/2$ and $\beta \ge 0$  and  any integer $\ell\ge 2$:
\begin{itemize}
\item[(i)]
if $\alpha/16+  \beta > 1$ and $\alpha/3+ \beta > 5/4$
and 
$$
  \max\left\{ 
\frac{\beta}{\alpha/16+  \beta - 1},  \frac{\beta}{\alpha/3+ \beta -5/4}\right\}  \le \vartheta_\ell 
$$
then we have 
$(\ell; \alpha,  \beta)   \in \fA_3$;

\item[(ii)] if $\alpha  +\beta > 3/2$ and
$$
 \frac{\beta}{\alpha  +\beta - 3/2}  \le \vartheta_\ell 
$$
then we have 
$(2, \ell; \alpha,  \beta)   \in \fA_4$;

\item[(iii)] if $\alpha/2  +\beta > 1$ and
$$
\frac{\beta}{\alpha/2  +\beta - 1}  \le \vartheta_\ell 
$$
then we have 
$(3, \ell; \alpha,  \beta)   \in \fA_4$;

\item[(iv)] if $\beta \ge 1/2$ and
$$
   \frac{\beta}{\beta -1/2}\le \vartheta_\ell 
$$
then we have 
$(4, \ell; \alpha,  \beta)   \in \fA_4$. 
\end{itemize}
\end{theorem}

We now consider the special case of $\alpha = \beta$.

\begin{cor}\label{cor:pPk}
For any  integer $\ell\ge 3$:
\begin{itemize}
\item[(i)]
For $\alpha = 16\vartheta_\ell /(17\vartheta_\ell -16)$ we have $(\ell; \alpha,  \alpha)   \in \fA_3$;

\item[(ii)] 
for $\alpha = 3\vartheta_\ell /(4\vartheta_\ell -2)$ 
we have $(2, \ell; \alpha,  \alpha)   \in \fA_4$;

\item[(iii)]
for $\alpha = 2\vartheta_\ell /(3\vartheta_\ell -2)$ 
we have $(3, \ell; \alpha,  \alpha)   \in \fA_4$;

\item[(iv)] 
for $\alpha = \vartheta_\ell /(2\vartheta_\ell -2)$ 
we have $(4, \ell; \alpha,  \alpha)   \in \fA_4$. 
\end{itemize}
\end{cor}

%
%
%

\section{Numerical Examples}
\label{sec: num ex}

First, we note that Corollary~\ref{cor:pPk}~(i), taken with $\ell=17$
improves (with respect to $\alpha$) the result from~\cite[Section~4]{Shp}, which we have presented 
in~\eqref{eq: 17 0.997}. However this  reduction  in the value of $\alpha$ is rather minor 
and is ``invisible'' at the level of numerical precision with which we present our results.  

On the other hand, for $k=2,3,4$ our improvements are more significant. 
For example, for $\ell=3$, we derive 
$$(2, 3; 0.905,  0.905), \, (3, 3; 0.864,  0.864),  \, (4, 3; 0.760,  0.760).
  \in \fA_4,$$
 In particular, recalling~\eqref{eq: Impl4-2}, we obtain
 $$
 (5; 0.905),\,  (6; 0.864)  \in \fA_2$$
each of which  improves~\eqref{eq: 6 15/16} and 
we also have 
 $$
(7; 0.760,  0.760) \in \fA_2
$$
which maybe compared with~\eqref{eq: k0 3/4}.

Furthermore, with $\ell = 4$ we obtain  
$$
(4, 4; 0.673,  0.673) \in \fA_4,
$$
and hence
$$
(8; 0.673) \in \fA_2
$$
which  improves~\eqref{eq: k0 3/4}.

We also see from Corollary~\ref{cor:pPk}~(iv) and~\eqref{eq: Impl4-2} that for any $\eta> 0$ 
there is some $k_\eta$ such that 
\begin{equation}
\label{eq: k0 1/2}
(k_\eta; 1/2 + \eta) \in \fA_2, 
\end{equation}
which is yet another improvement of~\eqref{eq: k0 3/4}.

We remark that here we have used the value of $\delta_3$ and  $\delta_4$ 
given by~\eqref{eq:delta234}, that
has been announced by Greaves~\cite[Equation~(1.4)]{Gre1}, 
see also~\cite[pp.~174--175]{Gre2}, however full details  of calculation have 
never been supplied (although there seems to be no reason to doubt the validity 
of these values).  However, even with slightly larger values, as those reported 
in~\cite{Gre0} our approach still leads to improvements of~\eqref{eq: 6 15/16} and~\eqref{eq: k0 3/4}.
  
\section{Notation}

Throughout the paper,  $p$ and $q$ always denote  prime 
numbers, while  $k$, $\ell$, $m$ and $n$ (in both the upper and
lower cases) denote positive integer 
numbers. 

We use $\Z_m$ to denote the residue ring modulo $m$. 

As we have mentioned, for an integer $\ell \ge 1$, we use $\cP_\ell$ to denote the set of 
integers that are products of at most $\ell$ primes. 

As usual, we  use $\pi(x)$ to denote the 
number of primes $p\le x$ and $P(n)$ to denote the
larget prime divisor of $n \ge 2$ (we also set $P(1) = 0$).

We fix a sufficiently large integer $m$ and for any integer $n$ with $\gcd(m,n)=1$  
we denote by $\overline n$ 
the multiplicative inverse of $n$ modulo $m$, that is, the unique integer $u$ defined by
the conditions
$$
nu \equiv 1 \pmod m \mand 1 \le u <m.
$$

We remark that once we write $\overline n$ we automatically assume 
that $\gcd(m,n)=1$. 

The implied constants in the symbols `$O$' and  `$\ll$' 
  may occasionally,
where obvious, depend on the small positive parameter $\eps$.
We recall that the notations $U = O(V)$  and   $U \ll V$  are all
equivalent to the assertion that the inequality $|U|\le cV$ holds for some
constant $c>0$.

Finally, the notation $z\sim Z$ means that $z$ must satisfy the inequality
$Z< z\leq 2Z$.


\section{Exponential sums with   reciprocals of primes}

For an integer $m\ge 2$, we  define the exponential function 
$\em = \exp(2 \pi i z/m)$ consider 
the exponential sums
$$
S_k(a;x) = \ssum_{\substack{p_1,\ldots, p_k  \le x \\ 
\gcd(p_1 \ldots p_k,m)=1}} \em\(a\ip_1 \ldots \ip_k \),\quad k =1.2, \ldots, 
$$
where $x\geq 2$ is a real number and $a$ is an integer.

Note that in~\cite{Shp} only the sums $S_1(a;x)$ have been employed together with 
the following  bound of Fouvry and  Shparlinski~\cite[Theorem~3.1]{FoSh}, 
\begin{equation}
\label{eq:S FS}
\left| S_{m}(a;x)\right| \le 
\( x^{15/16} + m^{1/4}x^{2/3}\) m^{o(1)}
\end{equation} 
uniformly for    $x \le m^{4/3}$ and integers $a$ with $\gcd(a,m)=1$. 

We note that the bound~\eqref{eq:S FS} extends a similar bound of 
Garaev~\cite[Theorem~1.1]{Gar1} from prime to composite moduli.
For convenience, we have dropped the condition
$x \ge m^{3/4}$  from~\cite[Theorem~3.1]{FoSh}
as for smaller values of $x$ the bound is trivial.

Here, since we study a modified question, we also make use of the sums  
$S_k(a;x)$ with $k =2,3,4$. 
 First we need the following simple statement:

\begin{lemma}
\label{lem:EnergyPrime} For any real $x\ge 2$ the number of solutions to the 
congruence
$$ \ip_1 \ip_2 \equiv   \iq_1 \iq_2 \pmod m, \qquad   p_1, p_2,  q_1, q_2 \le x,
$$
is  at most $x^{2+o(1)}\(x^2/m + 1\)$. 
 \end{lemma}
\begin{proof} Clearly we can rewrite this congruence as 
$$
p_1 p_2 \equiv   q_1 q_2 \not  \equiv 0 \pmod m.
$$
When a pair $(q_1, q_2)$ is chosen (trivially, in at most $x^2$ ways), this puts 
the product $p_1 p_2\le x^2$ in an arithmetic progression of the fork $a + km$ with $k=0, \ldots, K$, 
where $K = \fl{x^/m}$ (and $a \ne 0$). Hence, each out of 
the $K+1 \le x^2/m + 1$ elements
of this progression gives rise to at most $2$ pairs $(p_1, p_2)$.  Therefore, 
the number of such solutions is $2x^{2}\(x^2/m + 1\)$. This concludes the proof.
\end{proof}

Clearly we ignored some possible logarithmic savings in the proof of Lemma~\ref{lem:EnergyPrime} 
which do not affect our main results.

Our bounds rely on the following classical bound on bilinear sums, which dates 
back to Vinogradov~\cite[Chapter~6, Problem~14.a]{Vin} and has reappeared 
 in many forms since then. 

\begin{lemma}
\label{lem:BilinSums} For arbitrary sets $\cU, \cV \subseteq \Z_m$,  complex
numbers $\varphi_u$ and $\psi_v$ with
$$
\sum_{u \in \cU} |\varphi_u|^2  \le \varPhi  \mand
\sum_{v \in \cV}|\psi_v|^2 \le \varPsi,
$$
and an integer $a$ with $\gcd(a,m)=1$, we have
$$
\left|
\sum_{u\in \cU} \sum_{v \in \cV}
\varphi_u \psi_v \em(auv) \right| \le \sqrt{\varPhi\varPsi m }.
$$
\end{lemma}

\begin{lemma}
\label{lem:Sums S} For any real $x\ge m^{1/2}$, uniformly over integers $a$ with $\gcd(a,m)=f$, 
we have
\begin{align*}
&S_2(a;x)  \ll  x \(\frac{f x}{m}+1\) (m/f)^{1/2}, \\
&S_3(a;x)  \ll  x^{5/2}  \(\frac{f x}{m}+1\)^{1/2}  ,\\
&S_4(a;x)  \ll x^{4} (m/f)^{-1/2}. 
\end{align*}
\end{lemma}

 \begin{proof} 
%
%
 The bound on $S_2(a;x)$ is instant from 
 Lemma~\ref{lem:BilinSums},  if one uses the trivial bound 
 $$
 \varPhi = \varPsi \le \#\left\{\ip_1 \equiv \ip_2 \pmod {m/f},
 \ p_1, p_2\le x\right\} \le x \(fx/m+1\)
 $$ 
 (note that hereafter $\ip_1$ and  $\ip_2$ are computed modulo $m$ rather than modulo $m/f$
 but this does not affect the argument). 
 
 To estimate $S_3(a;x)$, we group $p_1$ and $p_2$ together, and again use Lemma~\ref{lem:BilinSums} with 
\begin{align*}
 \varPhi & =   \#\left\{ \ip_1 \ip_2 \equiv   \iq_1 \iq_2 \pmod {m/f}, \ p_1, p_2,  q_1, q_2 \le x\right \} \\
 & \ll x^{2} \(\frac{fx^2}{m}+1\)  \ll f m^{-1}x^{4} 
\end{align*}
by Lemma~\ref{lem:EnergyPrime} (where we have also used that $x \ge m^{1/2}$), 
and also, as before, with 
$$
\varPsi \le x \(fx/m+1\).
 $$ 
 
 Finally for $S_4(a;x)$, we group $p_1$ and $p_2$ as well as $p_3$ and $p_4$ together, and again  Lemma~\ref{lem:BilinSums} with 
$$
 \varPhi =  \varPsi   \ll fm^{-1}x^{4} 
$$
which concludes the proof. 
\end{proof}
 
Note that the assumption $x \ge m^{1/2}$ of Lemma~\ref{lem:Sums S}
is used only for the purpose of typographical simplicity of the bounds;
one can obtain more general statements which apply to any $x$. 

Let 
\begin{equation}
\label{eq:def T}
\begin{split}
T_k(a;x,y) = 
\#\{(p_1,\ldots p_k,v)~:~   p_1, & \ldots, p_k \le x, \ v \le y,\\
&  a\ip_1 \ldots \ip_k\equiv v \pmod m\}.
\end{split}
\end{equation} 
We recall our convention that in the definition of $T_k(a;x,y) $ we automatically assume 
that $ \gcd(p_j,m) =1$, $j=1, \ldots, k$.

We also use 
\begin{equation}
\label{eq:def Delta}
\Delta_k(a;x,y) = T_k(a;x,y)- \pi(x)^k y/m
\end{equation} 
to denote the deviation between $T_k(a;x,y) $ 
and its expected value.

\begin{lemma}\label{lem:Distr} For any real $x$  and $y$
and with $m \ge x\ge m^{1/2}$,  $m  \ge y \ge 1$ and integer $a$ 
with $\gcd(a,m)=1$, we have
\begin{align*}
&\Delta_1(a;x,y)  \le \( x^{15/16} + m^{1/4}x^{2/3}\) m^{o(1)},  \\
&\Delta_2(a;x,y)  \le x  m^{1/2+o(1)}, \\
&\Delta_3(a;x,y)  \le  x^{5/2+o(1)}   ,\\
&\Delta_4(a;x,y)  \le x^{4+o(1)} m^{-1/2}. 
\end{align*}
\end{lemma}

 \begin{proof} 
 The bound on $\Delta_1(a;x,y)$ is given by~\cite[Lemma~2]{Shp}
 (the fact that in~\cite{Shp} the prime $p$ is from a dyadic interval $[x/2,x]$ is
 inconsequential). 

Using the same standard   technique as in the proof of~\cite[Lemma~2]{Shp},  
in particular,   the Erd{\H o}s--Tur{\'a}n inequality, see~\cite{DrTi,KuNi},
we easily derive the other bounds  from Lemma~\ref{lem:Sums S}.

We remark, that this method gives 
$\(\pi(x) - \omega(m)\)^k y /m$ for the main term, where $\omega(m)$ is 
the number of distinct prime divisors of $m$. Since we trivially have 
$\omega(m) \ll \log m$ this difference gets absorbed in the error term. 
\end{proof}

\section{Sieving}

Here we collect some results of Greaves~\cite{Gre1,Gre2}
which underly our approach. 

Let $\cA =(a_1, \ldots, a_N)$ be a sequence of $N$ integers in 
the interval $[1,Y]$ for some real $Y >1$.
For an integer $d \ge 1$ we define $\cA(d)$ as a subsequence of 
$\cA$ consisting of elements $a_n$ with $d \mid a_n$. 

We say that $\cA$ has a {\it level of distribution $D$\/} if there is a multiplicative
function $\rho(d)$ and some constant $\kappa> 0$ such that for 
$$
r(d) = \left|\# \cA(d) -\rho(d) N\right|
$$
we have 
$$
\sum_{d \le D} r(d) \ll N^{1-\kappa}. 
$$
We remark the condition on the sum of error terms can be 
relaxed a little bit and generally instead of a power saving and logarithmic saving  
is sufficient. 

Given a level of distribution, we also define the {\it degree\/}
$$
g = \frac{\log Y}{\log D}. 
$$

We also recall the definition of the constants $\delta_\ell$ from Section~\ref{sec:main}.

Then in the above notation  by~\cite[Proposition~1, Chapter~5]{Gre2} we have:

\begin{lemma}\label{lem:Sieve} If for some integer $\ell\ge 2$ we have 
$$
g < \ell -\delta_\ell
$$
then for some element $a_n$ of $\cA$ we have $a_n \in \cP_\ell$. 
\end{lemma}

\section{Proof of Theorem~\ref{thm:pPkPl}}

%
%
%
%

 We fix $a$ with $\gcd(a,m)=1$ and for an integer $k\ge1$, consider the 
sequence $\cA_k$ consisting of  the smallest nonnegative residues
$$
v \equiv a\ip_1 \ldots \ip_k \pmod m
$$ 
for $p_1, \ldots, p_k \le x$ and such that 
these residues  satisfy $v \le y$.
In particular $\# \cA_k = T_k(a;x,y)$ as defined by~\eqref{eq:def T}.

As usual, for an integer $d\ge 1$ we denote by $\cA_k(d)$ the number of 
$v \in \cA_k$ with $d \mid v$.
Clearly $\#\cA_k(d)$ is number of solutions to the
congruence 
$$
a\ip_1 \ldots \ip_k   \equiv d u \pmod m, \qquad p_1, \ldots, p_k \le x, \ u \le y/d.
$$
Thus $\#\cA_k(d)  = 0$ if $\gcd(d,m)> 1$. 
Otherwise, that is, for $\gcd(d,m) = 1$, we have 
$$
\left| \#\cA_k(d) - \frac{\pi(x)^k  y}{dm}\right| \le  \Delta_k(a;x,y), 
$$
where $\Delta_k(a;x,y)$ is defined by~\eqref{eq:def Delta}. 

We now fix some sufficiently small $\varepsilon>0$.

Using  Lemma~\ref{lem:Distr}  we see that the levels of distribution $D_k$
of $\cA_k$, $k=1,2,3,4$, satisfy 
\begin{align*}
&D_1  \ge \min\{x^{1/16}ym^{-1-\varepsilon},  x^{1/3}ym^{-5/4-\varepsilon}\} \\
&D_2 \ge xy  m^{-3/2-\varepsilon}, \\
&D_3\ge  x^{1/2}y  m^{-1-\varepsilon}  ,\\
&D_4 \ge  y m^{-1/2-\varepsilon}, 
\end{align*}
provided that $m$ is large enough.

Since all elements of the sequences $\cA_k$ are in the interval $[1,y]$, their 
 degree satisfies 
$$
g_k \le  \frac{\log y}{\log D_k} . 
$$
Hence, for $x = m^{\alpha+\varepsilon}$ and $y = m^{\beta+\varepsilon}$ with   $1\ge \alpha \ge 1/2$ and $1 \ge \beta \ge 0$ 
we have 
\begin{align*}
&g_1  \le
\max\left\{\frac{\beta+\varepsilon}{\alpha/16 +\beta - 1}, 
\frac{\beta+\varepsilon}{\alpha/3 +\beta  - 5/4} \right\},
\\
&g_2 \le \frac{\beta+\varepsilon}{\alpha  +\beta - 3/2 }, \\
&g_3\le   \frac{\beta+\varepsilon}{\alpha/2  +\beta - 1}, \\
&g_4 \le   \frac{\beta+\varepsilon}{\beta - 1/2}, 
\end{align*}
provided that the denominators are positive. 
Recalling Lemma~\ref{lem:Sieve}, we conclude the proof. 

\section{Comments}

%
%
%
%
We remark that Theorem~\ref{thm:pPkPl} resembles results about 
the distribution of elements of $\cP_\ell$ in arithmetics 
progressions, see, for example, ~\cite[Theorem~25.8]{FrIw}. However, 
these results seem to be 
 completely independent and do not imply each other. 
 For example, the elements of $\cP_2$ produced by~\cite[Theorem~25.8]{FrIw}
 cannot be ruled out to be prime, and they are also relatively large compared 
 to the modulus).


\section*{Acknowledgement}

The author is grateful to John Friedlander for enlightening 
discussions on sieves and comments on an earlier version of this paper.   

Part of this work was also done when
the authors was visiting 
the Max Planck Institute for Mathematics, Bonn, and Fields Institute, Toronto, 
whose  generous 
support and hospitality are gratefully acknowledged. 

This work was also partially supported  by  
by the Australian Research Council Grant~DP140100118.

\end{document}